\documentclass[a4paper,12pt]{amsart}
\usepackage{amsmath}
\usepackage{amssymb}
\usepackage{mathrsfs}
\usepackage{enumerate}
\usepackage{ifthen}
\usepackage{graphicx}
\usepackage[T1]{fontenc} 
\usepackage{tabularx}
\usepackage{multirow}
\usepackage{color,soul}
\usepackage{tikz}
\usepackage{pgfplots} 
\usepackage{amsmath, amssymb}

\nonstopmode \numberwithin{equation}{section}
\newtheorem{thm}{Theorem}[section]
\newtheorem{cor}{Corollary}[section]
\newtheorem{lem}{Lemma}[section]

\newtheorem{conj}{Conjecture}

\theoremstyle{definition}

\newtheorem{rem}{Remark}[section]
\usepackage{hyperref}
\hypersetup{
    colorlinks=true,
    citecolor=blue,
    linkcolor=blue,
    urlcolor=blue
}

\newcounter{minutes}
\divide\time by 60
\newcounter{hours}
\multiply\time by 60
\addtocounter{minutes}{-\time}

\newcounter {own}
\def\theown {\thesection       .\arabic{own}}

{\qed\bigskip}

\newcounter{alphabet}
\newcounter{tmp}

\begin{document}
\title{A note on a subclass of  Bazilevi{\v{c}} functions}
\author{Lokenath Thakur}
\address{Lokenath Thakur, National Institute Of Technology Durgapur, West Bengal, India}
\email{lokenaththakur1729@gmail.com}

\def\thefootnote{}
\footnotetext{ {\tiny File:~\jobname.tex,
printed: \number\year-\number\month-\number\day,
          \thehours.\ifnum\theminutes<10{0}\fi\theminutes }
} \makeatletter\def\thefootnote{\@arabic\c@footnote}\makeatother
\begin{abstract}
    In this artcle, we introduce and investigate a subclass of  Bazilevi{\v{c}} functions, denoted by $\mathcal{B}_{\varphi_{A,B}}(\alpha^{(m)})$. We determine the Hardy space to which this subclass of  Bazilevi{\v{c}} functions belong to. Additionally, we provide a necessary condition for a particular case of this subclass. Finally, we obtain a sharp coefficient estimate for the functions associated with $\mathcal{B}_1(\alpha).$ 
\end{abstract}

\thanks{}

\maketitle
\pagestyle{myheadings}
\markboth{Lokenath Thakur }{A note on a subclass of Bazilev\'c functions}
\section{Introduction}
Let $\mathcal{H}$ be the class of all analytic functions in the unit disk $\mathbb{D}=\{~z\in\mathbb{C}:|z|<1\}$ and let $\mathcal{A}$ denote the subclass of $\mathcal{H}$ with functions of the form 
\begin{align}\label{A-05}
f(z)=z+\sum_{n=2}^\infty a_n z^n.
\end{align}

If \( g(z) \in \mathcal{A} \) is starlike (with respect to the origin) in \(\mathbb{D}\), \( P(z)\in \mathcal{A} \) is regular with \( \text{Re } P(z) > 0 \) in \(\mathbb{D}\), \(\beta\) is any real number and \(\alpha > 0\), then

\begin{align}\label{A-10}
f(z) = \left[ (\alpha + i\beta) \int_0^z P(\zeta)g^{\alpha}(\zeta)\zeta^{i\beta-1}d\zeta \right]^{1/(\alpha+i\beta)}
\end{align}
has been shown by  Bazilevi{\v{c}} \cite{bazilevic-1955} to be a regular and univalent function in \(\mathbb{D}\). The powers appearing in the formula are meant to be principal values. We denote by \( \mathcal{B}(\alpha, \beta, \mathcal{P}, g) \) the class of functions defined by \eqref{A-10}. Let $\mathcal{S}$, $\mathcal{S}^*$, $\mathcal{K}$, $\mathcal{C}~~\text{and}~~ \mathcal{P}$ denote the subclasses of univalent, starlike, convex, close-to-convex and Caratheodory functions, respectively (For these classes, see \cite{Duren} for instance). It is well known that the inclusion relations $\mathcal{K} \subset \mathcal{S}^* \subset \mathcal{C} \subset  \mathcal{B}(\alpha, \beta, \mathcal{P}, g) \subset \mathcal{S}$ are valid. So far,  Bazilevi{\v{c}} functions form the largest subclass of $\mathcal{S}$. For $\beta=0$, the class $\mathcal{B} (\alpha, \beta, \mathcal{P}, g)$ is simply denoted by $\mathcal{B}(\alpha,\mathcal{P},g)$.\\

For analytic functions \( g \) and \( h \) in \( \mathbb{D} \), \( g \) is said to be subordinate to \( h \) if there exists an analytic function \( \omega \) such that  
\[
\omega(0) = 0, \quad |\omega(z)| < 1 \quad \text{and} \quad g(z) = h(\omega(z)) \quad (z \in \mathbb{D}).
\]  
This subordination will be denoted by \( g \prec h \) or, conventionally, by  
\[
g(z) \prec h(z).
\]  
In particular, when \( h \) is univalent in \( \mathbb{D} \), \( g \prec h \) if and only if  
\[
g(0) = h(0) \quad \text{and} \quad g(\mathbb{D}) \subset h(\mathbb{D}).
\] \\

Let $\mathcal{M}$ be the class of zero-free analytic functions $\varphi$ in $\mathbb{D}$ with the normalization condition $\varphi(0) = 1$. Then, following the earlier work by Ma and Minda \cite{ma}, Kim and Sugawa \cite{kim} introduced the subclasses $\mathcal{S}^*(\varphi)$ and $\mathcal{K}(\varphi)$ of $\mathcal{A}$ as the sets of functions $f \in \mathcal{A}$ satisfying, respectively, the following subordination conditions:

\[
\frac{zf'(z)}{f(z)} \prec \varphi(z)
\]

and

\[
1 + \frac{zf''(z)}{f'(z)} \prec \varphi(z),
\]

for each $\varphi \in \mathcal{M}$. By definition, it is clear that

\[
f \in \mathcal{K}(\varphi) \iff zf' \in \mathcal{S}^*(\varphi).
\]

We also note that

\[
\mathcal{S}^*(\varphi) \subset \mathcal{S}^*(\psi) \quad \text{and} \quad \mathcal{K}(\varphi) \subset \mathcal{K}(\psi) \quad (\varphi \prec \psi).
\]

Next, for $A$ and $B$ such that $-1 \leq B < A \leq 1$, let us define the Möbius transformation $\varphi_{A,B}$ by

\[
\varphi_{A,B}(z) = \frac{1 + Az}{1 + Bz}.
\] \\

Let \( m \) be a positive integer, \(\alpha_1, \ldots, \alpha_m \in (0, +\infty)\) and set \(\gamma = \alpha_1 + \cdots + \alpha_m + i\beta\). If \( h \in \mathcal{P} \) for some  and \( g_1, \ldots, g_m \in \mathcal{S}^* \), then Kim and Sugawa \cite{sugawa} proved that the function \( f \) defined by  
\begin{align}\label{A-15}
    f(z) = \left[ \gamma \int_0^z g_1^{\alpha_1}(\zeta) \cdots g_m^{\alpha_m}(\zeta) h(\zeta) \zeta^{i\beta-1} d\zeta \right]^{1/\gamma}, 
\end{align}
belongs to \(\mathcal{B}(\alpha_1 + \cdots + \alpha_m, \beta,P,g)\).\\

For $\varphi \in \mathcal{M}$, we denote by $\mathcal{B}_{\varphi_{A,B}}(\alpha^{(m)},\mathcal{P},g)$ the class of functions $f \in \mathcal{A}$ which satisfy \eqref{A-15}, where $g_1, g_2,\cdots,g_m\in \mathcal{S}^{*}(\varphi_{A,B})$ (see, full details, \cite{janowski-1,janowski-2}) and $\alpha^{(m)}=\alpha_1 + \cdots + \alpha_m$.  In particular, if we put

\[
\varphi(z) = \varphi_{1,-1}(z) = \frac{1+z}{1-z},
\]
we easily see that

\[
\mathcal{B}_{\varphi_{1,-1}}(\alpha^{(1)},\mathcal{P},g) = \mathcal{B}(\alpha,\mathcal{P},g).
\]
Throughout this paper, for the sake of brevity, we shall make use of the simplified notation $\mathcal{B}_{\varphi_{A,B}}(\alpha^{(m)})$ for the class $\mathcal{B}_{\varphi_{A,B}}(\alpha^{(m)},\mathcal{P},g)$.\\

Let $\mathcal{H}^p (0 < p \leq \infty)$ denote the Hardy space of analytic functions $f(z)$ in $\mathbb{D}$, and define the integral means by
\[
M_p(r,f) = 
\begin{cases} 
\left( \frac{1}{2\pi} \int_0^{2\pi} |f(re^{i\theta})|^p  d\theta \right)^{1/p} & (0 < p < \infty) \\
\max\limits_{|z| \leq r} |f(z)| & (p = \infty).
\end{cases}
\tag{1.4}
\]

Then, by definition, an analytic function $f(z)$ in $\mathbb{D}$ belongs to the Hardy space $\mathcal{H}^p (0 < p \leq \infty)$ if
\[
\|f\|_p := \lim_{r \to 1^-} M_p(r,f) < \infty.
\]

For a function $f$ in the class $\mathcal{B}(\alpha,\mathcal{P},g)$, Miller \cite{miller} investigated the Hardy spaces to which $f(z)$ and $f'(z)$ belong. In the present paper, when $\varphi(z) = \varphi_{A,B}(z)$, we carry out the corresponding investigations for the function class $\mathcal{B}_{\varphi_{A,B}}(\alpha^{(m)})$. We also provide a necessary condition for $\mathcal{B}_{\varphi_{A,B}}(\alpha^{(m)})$ when $m=1$ and $g(z)=z$ and we denote this subclass by $\mathcal{B}_{1}(\alpha).$

\section{The Hardy space for the class $\mathcal{B}_{\varphi_{A,B}}(\alpha^{(m)})$}

Eenigenburg and Keogh \cite{keogh} have investigated the Hardy classes to which \( f(z) \) and \( f'(z) \) belong for \( f \in \mathcal{K} = \mathcal{B}(1, 0, \mathcal{P}, g) \). In this paper we carry out a similar investigation for the wider class $\mathcal{B}_{\varphi_{A,B}}(\alpha^{(m)})$. In what follows, we denote by \( k_{\theta}(z) \) any function of the form \( z (1 - e^{-i\theta} z)^{-2} \), where \( \theta\) a real constant. We require the following lemmas.

\begin{lem}\cite{Littlewood}\label{lem-1}
If \( f(z) \) is univalent, then \( f(z) \in\mathcal{H}^\lambda \), for all \( \lambda < 1/2 \).
\end{lem}

\begin{lem}\label{lem-2}
If \( P(z) \) is regular and \( \operatorname{Re} P(z) > 0 \) in $\mathbb{D}$, then \( P(z) \in\mathcal{H}^\lambda \) for all \( \lambda < 1 \).
\end{lem}

\begin{lem}\cite{keogh}\label{lem-3}
If \( g \in \mathcal{S}^* \) and \( g(z) \neq k_{a,\tau}(z) \) then there exists \( \epsilon = \epsilon(g) > 0 \) such that \( g(z) \in\mathcal{H}^{1/2+\epsilon} \).
\end{lem}

\begin{lem}\cite{hardy}\label{lem-4}
If \( f'(z) \in\mathcal{H}^\lambda (0 < \lambda < 1) \), then \( f(z) \in\mathcal{H}^{\lambda/(1-\lambda)} \).
\end{lem}

\begin{thm}
    If \( \mathcal{B}_{\varphi_{A,B}}(\alpha^{(m)}) \) with \( g(z) \neq k_{\theta}(z) \) then there exists \( \epsilon = \epsilon(f) > 0 \) such that \( f(z) \in\mathcal{H}^{1/2m+\epsilon} \).
\end{thm}
\begin{proof}
    Let $f(z)\in  \mathcal{B}_{\varphi_{A,B}}(\alpha^{(m)})$. From \eqref{A-15} we have

$$
f(z) = \left[ \alpha \int_0^z g_1^{\alpha_1}(\zeta) \cdots g_m^{\alpha_m}(\zeta) h(\zeta) \zeta^{i\beta-1} d\zeta \right]^{1/\alpha}
$$

 where $g_1, g_2,\cdots,g_m\in S^{*}(\varphi_{A,B})$, $h\in \mathcal{P}$ and $\alpha=\alpha_1 + \cdots + \alpha_m$.\\

 or
 \begin{align}\label{A-20}
      f'(z)=f^{1-\alpha}(z)g_1^{\alpha_1}(z) \cdots g_m^{\alpha_m}(z) h(z) z^{-1}
 \end{align}

If we let $F(z)=\left(\frac{f(z)}{z} \right)^{\alpha},$ then $F(z)$ is regular in $\mathbb{D}$ and satisfies 
$$
F'(z)=\frac{\alpha f^{1-\alpha}(z)f'(z)}{z^\alpha}-\frac{\alpha F(z)}{z}.
$$
Using \eqref{A-20} this becomes,
$$
F'(z)=\frac{\alpha g_1^{\alpha_1}(z) \cdots g_m^{\alpha_m}(z) h(z)}{z^{\alpha+1}}-\frac{\alpha F(z)}{z},
$$
and if $0\le\lambda\le1$, then for $ z=re^{i\theta}(0<r<1)$ we obtain 

\[
I(r) = \int_{0}^{2\pi} \left| F'(z) \right|^\lambda d\theta \leq \int_{0}^{2\pi} \left| \alpha \frac{ g_1^{\alpha_1}(z) \cdots g_m^{\alpha_m}(z)}{z^{\alpha+1}} P(z) \right|^{\lambda} d\theta
\]
\begin{align}\label{A-25}
    + \int_{0}^{2\pi} \left| \frac{\alpha F(z)}{z} \right|^{\lambda} d\theta \equiv I_1(r) + I_2(r).
\end{align}
Now
\[
I_2(r) = \int_{0}^{2\pi} \left| \frac{\alpha F(z)}{z} \right|^{\lambda} d\theta = \int_{0}^{2\pi} \left| \frac{\alpha}{z} \left[ \frac{f(z)}{z} \right]^{\alpha} \right|^{\lambda} d\theta
\]
\[
= \frac{\alpha^{\lambda}}{r^{(\alpha+1)\lambda}} \int_{0}^{2\pi} \left| f(z) \right|^{\alpha \lambda} d\theta,
\]
and since \( f(z) \) is univalent, by Lemma \ref{lem-1}, \(\lim_{r \to 1^{-}} I_2(r)\) exists provided that
\begin{align}\label{A-30}
    \alpha \lambda < 1/2.
\end{align}
To calculate \( I_1(r) \), we apply H\"older's inequality with conjugate indices $p_1,p_2,\cdots,p_{m+1}.$ We have

\begin{align*}
  \frac{r^{(\alpha+1)\lambda}}{\alpha^\lambda} I_1(r) &= \int_0^{2\pi} |g_1(z)|^{\alpha_1\lambda}\cdots|g_m(z)|^{\alpha_m\lambda}|P(z)|^{\lambda} d\theta\\
&\leq \left[ \int_0^{2\pi} |g_1(z)|^{\alpha\lambda p_1} d\theta \right]^{\frac{1}{p_{1}}} \cdots \left[ \int_0^{2\pi} |g_m(z)|^{\alpha\lambda p_m} d\theta \right]^{\frac{1}{p_{m}}} \left[ \int_0^{2\pi} |P(z)|^{\lambda p_{m+1}} d\theta \right]^{\frac{1}{p_{m+1}}} \\
&= J_1(r) \cdots J_{m+1}(r).  
\end{align*}
By Lemma \ref{lem-3}, $\lim_{r \to 1^{-1}}J_i$ exists if $\alpha \lambda p_i \le 1/2 +\epsilon(g_i)=1/2 +\epsilon_i$ for $i=1,2,\cdots,m$. Again by Lemma \ref{lem-2}, $\lim_{r \to 1^{-1}}J_{m+1}$ exists if $\lambda p_{m+1}<1.$ Therefore, we see that $\lim_{r \to 1^{-1}}I_r$ exists provided that $\lambda<1/2\alpha$, $2\alpha \lambda/(1+2\epsilon)\le 1/p_i$ for $ i=1,2,\cdots,m$ and $\lambda<1/p_{m+1}$, where $\sum_{i=1}^{m+1} 1/p_i =1$ and $p_i>1.$ Let $\epsilon^*= \max\{\epsilon_1, \epsilon_2,\cdots, \epsilon_m\}$.\\

Therefore, $2\alpha \lambda/(1+2\epsilon^*)\le 1/p_i$ for all $i=1,2,\cdots,m$ and $\lambda<1/p_{m+1}$. This system of inequalities gives $\lambda\le (1+2\epsilon^*)/(2\alpha m+1+2\epsilon^*)$, where $\epsilon^*$ is chosen sufficiently small so that $(1+2\epsilon^*)/(2\alpha m+1+2\epsilon^*)<1/2\alpha.$\\

Hence, we have shown that there exists $\epsilon>0$ such that $\lim_{r \to 1^{-1}}I_r$ exists if $\lambda\le(1+\epsilon)/(2\alpha m+1+\epsilon)$, i.e.
\begin{align}\label{A-35}
    F'(z) \in \mathcal{H}^{(1+\epsilon)/(2\alpha m+1+\epsilon).}
\end{align}

For $\lambda=( 1+\epsilon)/(2\alpha m+1+\epsilon),$ we have $0<\lambda<1$ and $\lambda/(1-\lambda)=1/2\alpha m$. By \eqref{A-35}
 and Lemma \ref{lem-4}, we conclude that 
 \begin{align}\label{A-40}
     F(z)\in \mathcal{H}^{(1+\epsilon)/2\alpha m.}
 \end{align}
Since \( f(z) = z[F(z)]^{1/\alpha} \), we have
\[
\int_0^{2\pi} |f(z)|^\lambda  d\theta = r^\lambda \int_0^{2\pi} |F(z)|^{\lambda/\alpha}  d\theta,
\]
and consequently from \eqref{A-40} we see that this integral will be bounded provided that \(\lambda/\alpha \leq (1+\epsilon)/2\alpha m \), i.e. \(\lambda \leq 1/2m + \epsilon/2\) i.e. \(f \in \mathcal{H}^{1/2m+\epsilon}\), and this completes the proof of the theorem.

The condition \(g(z) \neq k_{\theta}(z)\) in the theorem is essential. This may be seen by taking \(g(z) = k_{\theta}(z)\) and \(h(z) = zk'_{\theta}(z)/k_{\theta}(z)\). A simple calculation shows that for \(f(z) \in \mathcal{B}(\alpha, 0, zk'_{\theta}/k_{\theta}, k_{\theta})\) we have \(f(z) = k_{\theta}(z)\). Since
\[
\int_0^{2\pi} |k_{\theta}(z)|^{1/2}  d\theta = r^{1/2} \int_0^{2\pi} \frac{d\theta}{|1 - re^{i\theta}|} \geq \sqrt{2}r^{1/2} \log \frac{1}{1-r},
\]
we see that \(k_{\theta}(z) \notin \mathcal{H}^{1/2}\).

\end{proof}
\begin{cor}
    If $f \in \mathcal{B}_1(\alpha)$, with $g(z) \neq k_{\theta}(z)$, then there exists $\epsilon = \epsilon(f) > 0$ such that
\[
f(z) \in \mathcal{H}^{\frac{1}{2} + \epsilon}.
\]
\end{cor}
\section{Necessary condition for the class $\mathcal{B}_1(\alpha)$}

Let $f(z)\in \mathcal{A}$ be given by \eqref{A-05}. We say that $f \in \mathcal{C}(\beta)$ if and only if there exists a $\phi \in \mathcal{K}$ such that, 
$$
\left | \arg \frac{f'(z)}{\phi'(z)} \right|\le \frac{\beta \pi}{2} \quad (0<\beta <1)
$$
or equivalently,
$$
f'(z)=p^{\beta}(z)\phi'(z) \quad \textrm{where} \quad p(z)\in \mathcal{P}
$$
Now for $\phi(z)=z$ the class $ \mathcal{C}(\beta)$ reduces to a subclass, namely, $\mathcal{C}_I(\beta).$ We first notice that there exists a one-to-one corresponding between the classes $\mathcal{C}_I(\beta)$ and $ \mathcal{B}_1(\alpha)$. Let $g(z)\in \mathcal{B}_1(\alpha) ~~~~(\alpha>1)$ then there exists a function $h(z) \in \mathcal{P}$ such that 
\begin{align*}
    g(z)= \left[ \alpha \int_0^z t^{\alpha-1} h(t) dt \right]^{1/\alpha}.
\end{align*}
Now, consider the function $G(z)=\int_0^z (\frac{g(t)}{t})^{1-\frac{1}{\alpha }}(g'(t))^{\frac{1}{\alpha}}dt $. Then,
$$
G'(z)=\left(\left(\frac{g(z)}{z}\right)^{\alpha-1}g'(z)\right )^{\frac{1}{\alpha}}=h^{\frac{1}{\alpha}}(z),
$$
which indicates that $G(z)\in \mathcal{C}_I(\beta)$. 

Now let $f(z)\in \mathcal{C}_I(\beta)$, then there exists a function $p(z)\in \mathcal{P}$ such that $f'(z)=p^\beta(z)$. Now consider the function
$$
F(z)=\left[\frac{1}{\beta} \int_0^z\left(t f'(t) \right)^{\frac{1}{\beta}}t^{-1}dt\right]^{\alpha}=\left[\frac{1}{\beta} \int_0^z\left(t\right)^{\frac{1}{\beta}-1}f'^{\frac{1}{\beta}}(t)dt\right]^{\beta}=\left[\frac{1}{\beta} \int_0^z\left(t\right)^{\frac{1}{\beta}-1}p(z)dt\right]^{\beta},
$$
which indicates that $F(z)\in \mathcal{B}_1(\frac{1}{\beta})$. Therefore, we have been able to a make bridge between the classes $\mathcal{B}_1(\alpha)$ and $\mathcal{C}_I(\beta)$ when $\alpha>1 ~~\textrm{and} ~~ 0<\beta<1$. In the following theorem, we provide an important necessary condition for $\mathcal{B}_1(\alpha)$ when $\alpha>1.$ In order to state our result, we also introduce the notation  
\[P[\alpha, f](z) = 1 + \frac{zf''(z)}{f'(z)} + (\alpha - 1)\frac{zf'(z)}{f(z)}\]
for \(\alpha>1\) and \(f \in \mathcal{A}\).

\begin{thm}
If $f(z)\in \mathcal{B}_1(\alpha)$, then for $\alpha>1$ we get
$$
\int_{\theta_1}^{\theta_2}\operatorname{Re}\left(1 + \frac{zf''(z)}{f'(z)}\right) + (\alpha - 1)\operatorname{Re}\left(\frac{zf'(z)}{f(z)}\right)d\theta >-\pi ~~i.e.~~ \int_{\theta_1}^{\theta_2}\operatorname{Re}P[\alpha, f](z)d\theta >-\pi.
$$
\end{thm}

\begin{proof}
    If $f(z)\in \mathcal{B}_1(\alpha)$ with $\alpha>1$, then there exists a function $F(z) \in \mathcal{C}_I(\frac{1}{\alpha})$ such that
    $$
    F'(z)=\left(\left(\frac{f(z)}{z}\right)^{\alpha-1}f'(z)\right )^{\frac{1}{\alpha}}.
    $$
    Let $T(F,re^{i\theta})$ be the tangent to the curve $F(|z|=r)$ at $F(re^{i\theta})$.

    So,
$$
\arg T(F,re^{i\theta}) = \left(1-\frac{1}{\alpha}\right)\arg f(z)+ \frac{1}{\alpha}\arg zf'(z).
$$
Therefore,
\begin{align}\label{A-45}
    \frac{\partial }{\partial \theta}\arg T(F,re^{i\theta})=  \left(1-\frac{1}{\alpha}\right)\operatorname{Re}\left(1 + \frac{zf''(z)}{f'(z)}\right) + \frac{1}{\alpha}\operatorname{Re}\left(\frac{zf'(z)}{f(z)}\right)
\end{align}
Since $F(z)\in  \mathcal{C}_z(\frac{1}{\alpha})$, we have 
\begin{align}\label{A-50}
\int_{\theta_1}^{\theta_2}\frac{\partial }{\partial \theta}\arg T(F,re^{i\theta}) > \frac{-\pi}{\alpha}.
\end{align}
Hence from \eqref{A-45} and \eqref{A-50} we get,
$$
\int_{\theta_1}^{\theta_2}\operatorname{Re}\left(1 + \frac{zf''(z)}{f'(z)}\right) + (\alpha - 1)\operatorname{Re}\left(\frac{zf'(z)}{f(z)}\right)d\theta >-\pi ~~i.e.~~ \int_{\theta_1}^{\theta_2}\operatorname{Re}P[\alpha, f](z)d\theta >-\pi.
$$

\end{proof}
\section{a coefficient problem on $\mathcal{B}_1(\alpha)$}
The coefficient problem for the function class \( B_1(\alpha) \) was first investigated by Singh \cite{ram}, who established sharp bounds for the initial Taylor coefficients of functions in this class. For \( f(z)\in \mathcal{B}_1(\alpha) \) of the form \eqref{A-05}, Singh proved exact inequalities for \( |a_2| \), \( |a_3| \), and \( |a_4| \), with \( 0 \leq \alpha \leq 1 \) and \( \alpha \geq 1 \). Finding exact bounds for higher-order coefficients, particularly for \( n \geq 5 \), has proven more difficult. Recent progress includes the exact bound for \( |a_5| \) when \( 0 \leq \alpha \leq 1/2 \), obtained by Marjono et al. \cite{marjono}. These results suggest the following obvious conjectures.\\

\begin{conj}\label{conj-1}
Let \( 0 \leq \alpha \leq 1 \). Suppose that \( g \in S \) and is given by \( g(z) = z + b_2 z^2 + \cdots \) and  
\begin{align}\label{example}
    g'(z) \left[ \frac{g(z)}{z} \right]^{\alpha-1} = 1 + 2 \sum_{n=1}^\infty z^n.
\end{align}
If \( f \in \mathcal{B}_1(\alpha) \) of the form \eqref{A-05}, then for \( n \geq 2 \),  
\[
|a_n| \leq b_n.
\]
\end{conj}
\begin{conj}\label{conj-2}
Let \( \alpha \geq 1 \) and \( f \in\mathcal{B}_1(\alpha)\) be of the form \eqref{A-05}. Then, for \( n \geq 2 \),  
\[
|a_n| \leq \frac{2}{n-1 + \alpha}.
\]
\end{conj}
The conjectures stated above for the class $\mathcal{B}_1(\alpha)$ provide a natural starting point for further research and in light of these problems, we investigate the following coefficient problem.
\begin{thm}\label{thm-01}
    Let \( f(z) \in\mathcal{B}_1(\alpha)\) be of the form \eqref{A-05}, then for $\psi(z)=\left(\frac{f(z)}{z}\right)^{\alpha}$ with $\alpha>0$, we have
    $$
    |A_n|\le\frac{2\alpha}{n+\alpha},~~n=1,2,\cdots,
    $$
    where $\psi(z)=1+\sum_{n=1}^\infty A_n z^n$ and this estimate is sharp.
\end{thm}
\begin{proof}
    Let \( f(z) \in\mathcal{B}_1(\alpha)\). Then, we have,
    $$
   \operatorname{Re} \left( \frac{f(z)}{z} \right)^{\alpha-1}f'(z)>0,
    $$
    which is equivalent to,
    $$
    \left( \frac{f(z)}{z} \right)^{\alpha-1}f'(z)=p(z),
    $$
   where $p(z)=1+\sum_{n=1}^\infty p_n z^n$ has positive real part. \\

   Also we have,  $\psi(z)=\left(\frac{f(z)}{z}\right)^{\alpha}$ and on differentiation we get,
   $$
   \psi'(z)=\frac{\alpha}{z}.\left( \frac{f(z)}{z} \right)^{\alpha-1}f'(z)-\frac{\alpha}{z}.\psi(z).
   $$
   Thus,
   $$
   \psi'(z)=\frac{\alpha}{z}.p(z)-\frac{\alpha}{z}.\psi(z),
   $$
   which simplifies to,
   \begin{align}\label{A-500}
        z\psi'(z)+\alpha\psi(z)=\alpha p(z).
    \end{align}
   On comparing $n$-th coefficients in the above we get,
   $$
   (n+\alpha)A_n=\alpha p_n,
   $$
   which implies the desired inequality using Caratheodory bound $|p_n|\le2$.\\

   The inequality is sharp for the function 
   \begin{align}\label{example-2}
       G(z)=\left(\frac{g(z)}{z}\right)^{\alpha}=\frac{\alpha}{z^{\alpha}}.\int_0^z t^{\alpha-1} \frac{1+t}{1-t} ~dt,  \end{align} 
   where $g(z)$ is defined in \eqref{example}.

   \begin{rem}
   The function $G(z)$ defined in \eqref{example-2} satisfies the following differential subordination,
   \begin{align}\label{A-1000}
       zG'(z)+\alpha G(z) \prec \alpha \frac{1+z}{1-z},
   \end{align}
   and therefore combining \eqref{A-1000} and \eqref{A-500} we get,
   $$
   z\psi'(z)+\alpha \psi(z) \prec  zG'(z)+\alpha G(z).
   $$
   \end{rem}
\begin{rem}
    Theorem \ref{thm-01} implies the following coefficient domination relation,
    $$
   \psi(z)= \left(\frac{f(z)}{z}\right)^{\alpha}<<\left(\frac{g(z)}{z}\right)^{\alpha}=G(z).
    $$
    i.e.
    $$
    |a_n(\psi(z))|\le \frac{2\alpha}{n+\alpha},~~~\alpha>0.
    $$
\end{rem}
This establishes the validity of Conjecture \ref{conj-1} and Conjecture \ref{conj-2} for all integer values of $\alpha$. However, the validity of these conjectures for non-integer $\alpha>0$ remains an open problem. 
\end{proof}
\vspace{4mm}

\noindent\textbf{Data availability:}
Data sharing is not applicable to this article as no data sets were generated or analyzed during the current study.

\vspace{1mm}
\noindent\textbf{Acknowledgement:}
The author thanks the CSIR for the financial support through CSIR-SRF Fellowship ( File no. : 09/0973(13731)/2022-EMR-I ).

\end{document}